\documentclass[11pt]{amsart}
\usepackage{amsmath,amssymb}
\usepackage[dvips]{graphicx}

\usepackage{fullpage}
\usepackage{enumerate}

\def\COMMENT#1{}
\let\COMMENT=\footnote
\def\TASK#1{}

\def\noproof{{\unskip\nobreak\hfill\penalty50\hskip2em\hbox{}\nobreak\hfill%
        $\square$\parfillskip=0pt\finalhyphendemerits=0\par}\goodbreak}
\def\endproof{\noproof\bigskip}
\newdimen\margin   
\def\textno#1&#2\par{%
    \margin=\hsize
    \advance\margin by -4\parindent
           \setbox1=\hbox{\sl#1}%
    \ifdim\wd1 < \margin
       $$\box1\eqno#2$$%
    \else
       \bigbreak
       \hbox to \hsize{\indent$\vcenter{\advance\hsize by -3\parindent
       \sl\noindent#1}\hfil#2$}%
       \bigbreak
    \fi}

\def\eps{\varepsilon}

\def\a{\alpha}

\def\d{\delta}
\def\g{\gamma}
\def\th{\theta}

\def\B{\mathcal{B}}
\def\oB{ \overline{\mathcal{B}} }

\def\dB{\mathcal B _{n,k}}
\def\doB{\overline{\mathcal B} _{n,k}}

\newtheorem{firstthm}{Proposition}
\newtheorem{thm}[firstthm]{Theorem}
\newtheorem{prop}[firstthm]{Proposition}
\newtheorem{lemma}[firstthm]{Lemma}
\newtheorem{cor}[firstthm]{Corollary}

\newtheorem{conj}[firstthm]{Conjecture}
\newtheorem{claim}[firstthm]{Claim}

\begin{document}
\title{A note on perfect matchings in uniform hypergraphs}
\author{Andrew Treglown and Yi Zhao}
\thanks{The first author is supported by EPSRC grant EP/M016641/1. The second author is partially supported by NSF grant DMS-1400073.}

\date{\today}

\begin{abstract}
We determine the \emph{exact} minimum $\ell$-degree threshold for perfect matchings in $k$-uniform hypergraphs when the corresponding threshold for perfect fractional matchings is significantly less than $\frac{1}{2} \binom{n}{k- \ell}$. This extends our previous results 
that determine the minimum $\ell$-degree thresholds for perfect matchings in $k$-uniform hypergraphs for all $\ell\ge k/2$ and provides two new (exact) thresholds: $(k,\ell)=(5,2)$ and $(7,3)$.

\end{abstract}

\maketitle
\section{Introduction}
A \emph{perfect matching} in a hypergraph $H$ is a collection of vertex-disjoint edges of $H$ which cover the vertex set $V(H)$ of $H$.
Given a $k$-uniform hypergraph $H$ with an $\ell$-element vertex set $S$ (where $0 \leq \ell \leq k-1$) we define
$d_H (S)$ to be the number of edges containing $S$. The \emph{minimum $\ell$-degree $\delta _{\ell}
(H)$} of $H$ is the minimum of $d_H (S)$ over all $\ell$-element sets of vertices in $H$.
In recent years the problem of determining the minimum $\ell$-degree threshold that ensures a perfect matching in a $k$-uniform hypergraph has received much attention
(see e.g.~\cite{ko1, rrs1, pik, rrs, hps, mark, KOTmatch,   khan1, khan2, afh, zhao, czy,  zhao2, mycroft, tim, jie}).
See~\cite{rrsurvey} for a survey on matchings (and Hamilton cycles) in hypergraphs.

Suppose that $\ell, k, n \in \mathbb N$ such that $k \geq 3$, $\ell \le k-1$ and $k$ divides $n$.
Let $m_{\ell}(k, n)$ denote the smallest integer $m$ such that every $k$-uniform hypergraph $H$ on $n$ vertices with $\delta _{\ell}(H)\ge m$ contains a perfect matching.
The conjectured value of $m_{\ell}(k, n)$ comes from two types of constructions.
The first type of constructions are referred to as \emph{divisibility barriers}.
Given a set $V$ of $n$ vertices with a partition $A, B$,
let $E_{\text{odd}}(A, B)$ ($E_{\text{even}}(A, B)$) denote the family of all $k$-element subsets of $V$ that intersect $A$ in an odd (even) number of vertices.
Define $\mathcal B_{n,k}(A,B)$ to be the $k$-uniform hypergraph with vertex set $V$ and edge set $E_{\text{odd}} (A,B)$.  Note that the complement $\overline{\mathcal B}_{n,k} (A,B)$ of $\mathcal B_{n,k} (A,B)$ has edge set $E_{\text{even}} (A,B)$.
Define $\mathcal H_{\text{ext}} (n,k)$ to be the collection of the following hypergraphs. First, $\mathcal H_{\text{ext}} (n,k)$ contains all hypergraphs $\doB (A,B)$ where $|A|$ is odd. Second, if $n/k$ is odd
then $\mathcal H_{\text{ext}} (n,k)$ also contains all hypergraphs $\dB (A,B)$ where $|A|$ is even; if $n/k$ is even then $\mathcal H_{\text{ext}} (n,k)$ also contains all hypergraphs $\dB (A,B)$ where $|A|$ is odd.
It is easy to see that no hypergraph in $\mathcal H_{\text{ext}} (n,k)$ contains a perfect matching.
Define $\delta (n,k, \ell)$ to be the maximum of the minimum $\ell$-degrees among all the hypergraphs in  $\mathcal H_{\text{ext}} (n,k)$. Note that
$\delta (n,k, \ell)= (1/2+o(1))\binom{n-\ell}{k-\ell}$ but the general formula of
$\delta (n,k, \ell)$ is unknown
(see more discussion in \cite{zhao}).

The other type of extremal constructions are referred to as \emph{space barriers}.
Let $H^*(n,k)$ be the $k$-uniform hypergraph on $n$ vertices whose vertex set is
partitioned into two vertex classes
$A$ and $B$ of sizes  $n/k-1$ and $(1-1/k)n +1$ respectively and whose edge set consists precisely
of all those edges with at least one endpoint in $A$. Then $H^*(n,k)$ does not have a perfect matching and
$\delta _{\ell} (H^*(n,k)) =  \binom{n-\ell}{k-\ell}-\binom{(1-1/k)n -\ell +1}{k-\ell}\approx \left(1-\left(\frac{k-1}{k}\right) ^{k-\ell} \right ) \binom{n-\ell}{k-\ell}$.

An asymptotic version of the following conjecture appeared in~\cite{hps, survey} and the minimum vertex degree version was stated  in~\cite{KOTmatch}.

\begin{conj}\label{generalconj}
Let $k, \ell \in \mathbb N$ such that $\ell \leq k-1$. Then for sufficiently large $n\in k\mathbb{N}$,
\[
m_{\ell}(k, n)= \max \left \{ \delta (n,k,\ell), \ \binom{n-\ell}{k-\ell}-\binom{(1-1/k)n -\ell +1}{k-\ell}\right \} + 1.
\]
\end{conj}

Note that for all $1\leq \ell \leq k-1$,
\[
\left( \frac{k-1}{k} \right) ^{k-\ell} < \left ( \frac{1}{e} \right)^{1-\frac{\ell}{k}} \ \ \text{ and } \ \
1- \left (\frac{k-1}{k} \right) ^{k \ln 2} \rightarrow \frac{1}{2} \ \ \text{
as } \ \ k \rightarrow \infty,
\]
where $\ln$ denotes the natural logarithm function.
Thus, for $1\ll k\ll n$, if $\ell$ is significantly bigger than $(1-\ln 2)k \approx 0.307k$
then  $\delta(n,k, \ell) > \binom{n-\ell}{k-\ell}-\binom{(1-1/k)n -\ell +1}{k-\ell}$.
On the other hand, if $\ell$ is smaller than $(1-\ln 2)k$ then $\delta(n,k, \ell) < \binom{n-\ell}{k-\ell}-\binom{(1-1/k)n -\ell +1}{k-\ell}$ for sufficiently large $n$.

Conjecture~\ref{generalconj} has been proven in a number of special cases. Indeed, R\"odl, Ruci\'nski and Szemer\'edi~\cite{rrs} proved the conjecture for $\ell= k-1$.
The authors~\cite{zhao, zhao2} generalized this result by showing $m_{\ell}(k, n)=  \delta (n,k,\ell)+1$ for all $k/2 \leq \ell \leq k-1$
(independently Czygrinow and Kamat~\cite{czy} proved this for $(k, \ell)=(4,2)$).
In the case when $(k, \ell)=(3,1)$, Conjecture~\ref{generalconj} was confirmed by K\"uhn, Osthus and Treglown~\cite{KOTmatch} and independently
Khan~\cite{khan1}. Khan~\cite{khan2} also resolved the case when $(k, \ell)=(4,1)$. Alon, Frankl, Huang, R\"odl, Ruci\'nski and Sudakov~\cite{afh}  determined $m_{\ell}(k, n)$ asymptotically in the case when
$(k, \ell)= (5, 1)$, $(5, 2)$, $(6, 2)$, and  $(7, 3)$. Other than these results, no other
asymptotic or exact  results are known (the best known general bounds are due to K\"uhn, Osthus and Townsend~\cite{tim}).

A connection between  $m_{\ell}(k, n)$ and the minimum $\ell$-degree that forces a \emph{perfect fractional matching} was discovered in \cite{afh}.
Let $H$ be a $k$-uniform hypergraph on $n$ vertices. A \emph{fractional matching} in $H$ is a function $w: E(H) \rightarrow [0,1]$ such that for each $v \in V(H)$ we have that $\sum _{e \ni v} w(e)\leq1$. Then $\sum _{e \in E(H)} w(e)$ is the \emph{size} of $w$. If the size of the largest fractional matching $w$ in $H$ is $n/k$ then we say that $w$ is a \emph{perfect fractional matching}.
Given $k,\ell \in \mathbb N$ such that $\ell \leq k-1$, define $c^*_{k,\ell}$ to be the smallest number $c$ such that every $k$-uniform hypergraph $H$ on $n$ vertices with $\delta _{\ell} (H) \geq (c +o(1)) \binom{n-\ell}{k-\ell}$ contains a perfect fractional matching. It is easy to see that
the hypergraph $H^* (n,k)$ defined earlier contains no perfect fractional matching. Thus $c^*_{k,\ell}\ge 1-\left(\frac{k-1}{k}\right) ^{k-\ell}$.
Alon et al.  \cite[Theorem 1.1]{afh} showed that for fixed $k, \ell$, as $n\in  k\mathbb{N}$ tends to infinity,
\begin{equation}
\label{eq:afh}
m_{\ell}(k, n)= \left( \max \left \{ \frac12, \ c^*_{k, \ell} \right \} + o(1) \right) \binom{n-\ell}{k-\ell}.
\end{equation}
Furthermore, in~\cite{afh} the authors conjectured that $c^*_{k,\ell}= 1-\left(\frac{k-1}{k}\right) ^{k-\ell}$ and confirmed this for $\ell \ge k-4$.
Together with \eqref{eq:afh}, this gives the aforementioned (asymptotic) results on  $m_{\ell}(k, n)$ for $(k, \ell)=  (5, 1)$, $(5, 2)$, $(6, 2)$ and $(7, 3)$.

In this note we prove the following refinement of \eqref{eq:afh}.
\begin{thm} \label{mainthm}
Fix $k, \ell \in \mathbb N$ with $\ell \leq k-1$ and let $n\in  k\mathbb{N}$. Then
\[
m_{\ell}(k, n)= \max \left \{ \delta (n,k,\ell)+1, \ (c^*_{k, \ell}  + o(1)) \binom{n-\ell}{k-\ell} \right \}.
\]
\end{thm}
Although it looks like a small improvement, Theorem~\ref{mainthm} enables us to determine $m_{\ell}(k, n)$ \emph{exactly} whenever $c^*_{k, \ell}< 1/2$.  
A recent result of  K\"uhn, Osthus and Townsend~\cite[Theorem 1.9]{tim} showed that
\[
c^*_{k, \ell} \le \frac{k- \ell}{k} - \frac{k - \ell - 1}{k^{k-\ell}}
\]
for all $\ell \le k -2$. Together with $c^*_{k, k-1}= 1/k$ (see \cite{rrs1}), this implies that $c^* _{k, \ell} < 1/2$ for all $k/2 \le \ell \le k -1$. Consequently Theorem~\ref{mainthm} implies the aforementioned results of \cite{zhao, zhao2}: $m_{\ell}(k, n)= \delta (n,k, \ell)+1$ for all $k/2 \leq \ell \leq k-1$.

Furthermore, in~\cite{afh} it was shown that $c^*_{5,2}=61/125<1/2$ and $c^*_{7,3}=1105/2401<1/2$. Therefore an immediate consequence of Theorem~\ref{mainthm}  is the following corollary. 

\begin{cor}\label{cor1}
Suppose that $(k,\ell)=(5,2)$ or $(7,3)$. Then $m_{\ell}(k, n)= \delta (n,k, \ell)+1$ for sufficiently large $n$.
\end{cor}

After this paper was submitted, Han~\cite{hanerdos} showed that $c^* _{k,\ell}<1/2$ in the case when $0.42k\leq \ell <k/2$ or $(k,\ell)=(12,5),(17,7)$. Thus, together with Theorem~\ref{mainthm} this resolves Conjecture~\ref{generalconj} in these cases.

 Let us highlight the ideas behind the proof of Theorem~\ref{mainthm}. It is informative to first recall the proof of \eqref{eq:afh}, which is an application of the absorbing method. The authors of \cite{afh} first applied a lemma of H\`an, Person and Schacht~\cite[Lemma 2.4]{hps}, which states that every $k$-uniform hypergraph $H$ with $\delta_{\ell}(H)\ge (1/2 + o(1)) \binom{n}{k - \ell}$ contains a small matching $M$ that can absorb any vertex set $W\subset V(H)\setminus V(M)$ of size much smaller than $M$ (that is, there is matching in $H$ on $V(M)\cup W$ exactly). Next they found an almost perfect matching in $H[V(H)\setminus V(M)]$ by first finding a fractional matching and then converting it to an integer matching. This immediately provides the desired perfect matching of $H$ because $M$ can absorb all uncovered vertices in $V(H)\setminus V(M)$. In order to prove Theorem~\ref{mainthm}, we prove a result stronger than \cite[Lemma 2.4]{hps}, Theorem~\ref{thm:abs}, which implies that every $k$-uniform hypergraph $H$ with $\delta_{\ell}(H)\ge (1/2 - o(1)) \binom{n}{k - \ell}$ either contains the aforementioned $M$ or looks like a hypergraph in $\mathcal H_{\text{ext}} (n,k)$. If $H$ contains $M$, then we proceed as in \cite{afh} (except that we apply a lemma from \cite{tim} when converting a fractional matching to an integer matching); if $H$ looks like a hypergraph in $\mathcal H_{\text{ext}} (n,k)$ and $\delta_{\ell}(H)\ge \delta(n, k, \ell) + 1$, then we obtain a perfect matching by applying \cite[Theorem 4.1]{zhao}.

Theorem~\ref{thm:abs} is the main contribution of this note -- it is stronger than two absorbing theorems in our previous papers \cite[Theorem 3.1]{zhao} and \cite[Theorem 3.1]{zhao2}, in which we assume that $\ell \ge k/2$. The proof of Theorem~\ref{thm:abs} is actually shorter than those of the two previous absorbing theorems because 1) we use a different absorbing structure which allows us to apply a lemma from \cite{lo}; 2) when proving Lemma~\ref{lem:main}, we avoid using auxiliary hypergraphs and obtain the structure of $H$ by considering the neighborhoods of the vertices of $H$ directly.

\vspace{2mm}
\noindent\textbf{Notation:}
Given a set $X$ and  $r\in \mathbb N$, we write $\binom{X}{r}$ for the set of all $r$-element subsets  of $X$.
Let $H$ be a $k$-uniform hypergraph.
We write $V(H)$ for the vertex set and $E(H)$ for the edge set  of $H$.
Define $e(H):=|E(H)|$. Given ${v}\in V(H)$,  we write $N_H(v)$  to denote the \emph{neighborhood of $v$}, that is, the family of those $(k-1)$-subsets of $V(H)$ which, together with $v$, form an edge in $H$. Given $X \subseteq V(H)$, we write $H[X]$ for the \emph{subhypergraph of $H$ induced by $X$},  namely, $H[X] := (X, E(H)\cap \binom{X}{k})$. We denote the \emph{complement of $H$} by $\overline{H}$. That is, $\overline{H} := (V(H), \binom{V(H)}{k}\setminus E(H))$.
Suppose that $n,k \in \mathbb N$. When $|A|=\lfloor n/2 \rfloor$ and $|B|=\lceil n/2 \rceil$, we define $\mathcal B_{n,k}:=\mathcal B_{n,k}(A,B)$ and $\overline{\mathcal B}_{n,k} := \overline{\mathcal B}_{n,k}(A, B)$.

We will often write $0<a_1 \ll a_2 \ll a_3$ to mean that we can choose the constants
$a_1,a_2,a_3$ from right to left. More
precisely, there are increasing functions $f$ and $g$ such that, given
$a_3$, whenever we choose some $a_2 \leq f(a_3)$ and $a_1 \leq g(a_2)$, all
calculations needed in our proof are valid.
Hierarchies with more constants are defined in the obvious way.

\section{Proof of Theorem~\ref{mainthm}}
The lower bound for $m_{\ell}(k, n)$ in Theorem~\ref{mainthm} follows from the definitions of $\delta(n, k, \ell)$ and $c^*_{k, \ell}$ immediately. The following (more general) result provides the desired upper bound for $m_{\ell}(k, n)$. 
\begin{thm}\label{main2}
Given any $\th >0$, $k, \ell, \ell ' \in \mathbb N$ where $1\leq \ell, \ell ' \leq k-1$ there is an $n_0 \in \mathbb N$ such that the following holds. Let $n \geq n_0$ where $k$ divides $n$.
 If $H$ is a $k$-uniform hypergraph on $n$ vertices with
\begin{equation}
\label{eq:2deg}
\delta _{\ell} (H)  >  \delta (n,k,\ell) \ \ \text{ and } \ \ \delta_{\ell '} (H) > (c^*_{k,\ell'} +\th) \binom{n-\ell '}{k-\ell '} ,
\end{equation}
then $H$ contains a perfect matching.
\end{thm}

The proof of Theorem~\ref{main2} splits into \emph{extremal} and \emph{non-extremal} cases, the former case being when $H$ looks like an element of $\mathcal H_{\text ext} (n,k)$. To make this precise we introduce more notation.
Let $\eps>0$. Suppose that $H$ and $H'$ are $k$-uniform hypergraphs on $n$ vertices. We say that $H$ is \emph{$\eps$-close to $H'$} if $H$ becomes a copy of $H'$ after adding and deleting at most $\eps n^k$ edges. More precisely,  $H$ is $\eps$-close to $H'$ if there is an isomorphic copy $\tilde{H}$ of $H$ such that $V(\tilde{H}) = V(H')$ and $|E(\tilde{H})\triangle E(H')| \le \eps n^k$.

Our proof of the non-extremal case uses the absorbing method. Given a $k$-uniform hypergraph $H$,  a set $S\subseteq V(H)$ is called an \emph{absorbing set for $Q\subseteq V(H)$}, if both $H[S]$ and $H[S\cup Q]$ contain perfect matchings. In this case, if the matching covering $S$ is $M$, we also say \emph{$M$ absorbs $Q$}.

Our main result, Theorem~\ref{thm:abs}, extends \cite[Theorem 3.1]{zhao2}. It ensures that if $H$ is as in Theorem~\ref{main2} then  $H$ contains a small absorbing matching or $H$ is close to one of $\mathcal B_{n,k}$ and  $\overline{\mathcal B}_{n,k}$. We postpone its proof to the next subsection.

\begin{thm}\label{thm:abs}
Given any $\eps >0$ and integer $k \geq 2$, there exist $0< \a, \xi < \eps$
and $n_0 \in \mathbb N$ such that the following holds.  Suppose that $H$ is a $k$-uniform hypergraph on $n \geq n_0$ vertices. If
$$\delta _{1} (H) \geq \left( \frac{1}{2}-\a \right) \binom{n - 1}{k- 1}$$
then $H$ is $\eps$-close to $\mathcal B_{n,k}$ or $\overline{\mathcal B}_{n,k}$, or $H$ contains
a matching $M$ of size $|M| \le \xi n/k$ that absorbs any set $W\subseteq V(H) \setminus V(M)$ such that $|W| \in k\mathbb{N}$ with $|W| \le \xi^2 n$.
\end{thm}

The next result from~\cite{zhao} ensures a perfect matching when our hypergraph $H$ is close to one of the
extremal hypergraphs $\mathcal B_{n,k}$ and $\overline{\mathcal B}_{n,k}$.

\begin{thm}\cite[Theorem 4.1]{zhao}\label{extthm}
Given $1\le \ell \le k-1$, there exist $\eps > 0$ and $n_0 \in \mathbb{N}$ such that the following holds. Suppose that $H$ is a $k$-uniform hypergraph on $n\ge n_0$ vertices such that $n$ is divisible by $k$. If  $\delta_{\ell} (H) >  \d(n, k, \ell)$ and $H$ is $\eps$-close to $\mathcal B_{n,k}$ or $\overline{\mathcal B}_{n,k}$, then $H$ contains a perfect matching.
\end{thm}

The final tool required for the proof of Theorem~\ref{main2} is a weaker version of Lemma 5.6 in~\cite{tim}.

\begin{lemma}\cite{tim} \label{lem:frac}
Let $k \geq 2$ and $1 \leq \ell \leq k-1$ be integers, and let $\eps >0$. Suppose that for some $b,c \in (0,1)$ and some $n_0 \in \mathbb N$, every $k$-uniform hypergraph $H$ on $n \geq n_0$ vertices with $\delta _{\ell} (H) \geq cn^{k-\ell}$ has a fractional matching of size $(b+\eps)n$. Then there exists an $n_0 ' \in \mathbb N$ such that any $k$-uniform hypergraph $H$ on $n \geq n'_0$ vertices with $\delta _{\ell} (H) \geq (c+ \eps) n^{k-\ell}$ contains a matching of size at least $bn$.
\end{lemma}

\medskip
\noindent
\textbf{Proof of Theorem~\ref{main2}.}
Choose $\eps >0$ from Theorem~\ref{extthm}. We may additionally assume that $\eps \ll \theta ,1/k$.
Let $0<\a, \xi<\eps$  be as in Theorem~\ref{thm:abs}.
Let $n$ be sufficiently large and divisible by $k$. Assume that $H$ is a $k$-uniform hypergraph on $n$ vertices satisfying \eqref{eq:2deg}.
Since $\delta _{\ell} (H)  >  \delta (n,k,\ell) = (1/2 - o(1)) \binom{n- \ell}{k- \ell}$, it follows that $\delta_1(H)\ge (1/2 - \a) \binom{n-1}{k-1}$. By Theorem~\ref{thm:abs}, $H$ is $\eps$-close to $\mathcal B_{n,k}$ or $\overline{\mathcal B}_{n,k}$, or $H$ contains a matching $M$ of size $|M| \le \xi n/k$ that absorbs any set $W\subseteq V(H) \setminus V(M)$ satisfying $|W| \in k\mathbb{N}$ with $|W| \le \xi^2 n$.
In the former case, since $\delta _{\ell} (H)  >  \delta (n,k,\ell)$, Theorem~\ref{extthm} provides a perfect matching in $H$. In the latter case, set $c^* := c^*_{k, \ell'}$, $H':=H[V(H)\setminus V(M)]$ and $n_1:=|V(H')|$.
Since $|V(M)| \le \xi n$,
\[
\delta _{\ell'} (H') \geq \delta _{\ell'} (H) - |V(M)|\binom{n - \ell' - 1}{k-\ell'-1} \geq (c^* + \th - \xi k) \binom{n- \ell'}{k- \ell'} > (c^*  /(k-\ell ')! +2\xi ^2 )n_1 ^{k- \ell'},
\]
where the last inequality follows since $\xi \ll \theta , 1/k$.
Let $c:= c^*/(k- \ell')! + \xi ^2$. By the definition of $c^*$,  for sufficiently large $\tilde{n}$, every $k$-uniform hypergraph $F$ on $\tilde{n}$ vertices with $\delta_{\ell'}(F)\ge c \tilde{n} ^{k- \ell'}$ contains a perfect fractional matching.
Applying Lemma~\ref{lem:frac} with $ \xi ^2/k$ and $(1 - \xi ^2)/k$ playing the roles of $\eps$ and $b$ respectively, we conclude that $H'$ contains a matching $M'$ of size at least $(1- \xi^2) n_1/k$. Let $W$ be the uncovered vertices  of $H'$. Then $|W|\le \xi^2 n$. We finally absorb $W$ using the absorbing property of $M$.
\endproof

\subsection{Proof of Theorem~\ref{thm:abs} }
The proof of  Theorem~\ref{thm:abs} follows from the following three lemmas.
Lemma~\ref{lo} is a special case of \cite[Lemma~1.1]{lo} and  gives a sufficient condition for a hypergraph $H$ to contain a small matching that absorbs \emph{any} much smaller set of vertices from $H$.

\begin{lemma}\cite{lo} \label{lo}
Let $k\in \mathbb N$ and  $\gamma ' >0$.
Then there exists an $n_0 \in \mathbb N$ such that the following holds. Suppose that $H$ is a $k$-uniform hypergraph
on $n \geq n_0$ vertices so that, for any $x,y \in V(H)$, there are at least $\gamma ' n^{2k-1}$ $(2k-1)$-sets $X \subseteq V(H)$ such that both $H[X \cup \{x\}]$ and $H[X \cup \{y\}]$ contain perfect matchings.
Then $H$ contains a matching $M$ so that
\begin{itemize}
\item $|M|\leq (\gamma '/2)^k n/(8k^2(k-1))$;
\item $M$ absorbs any $W \subseteq V(G) \setminus M$ such that $|W| \in k \mathbb N$ and  $|W|\leq (\gamma '/2)^{2k} n/(128k(k-1)^2) $.
\end{itemize}
\end{lemma}

\begin{lemma}\label{absorbing}
Let $k\in \mathbb N$ and  $0< \gamma ' \ll \gamma \ll 1/k$. There exists an $n_0 \in \mathbb N$ such that the following holds.
Let $H=(V,E)$ be a $k$-uniform hypergraph on $n \geq n_0$ vertices. Suppose that for every $x,y \in V$
at least one of the following conditions holds.
\begin{itemize}
\item[(i)] $|N_H (x) \cap N_H (y) | \geq \gamma n^{k-1}$;
\item[(ii)] There exists at least $\gamma n$ vertices $z \in V$ such that
$|N_H(x) \cap N_H(z)| \geq \gamma n^{k-1}$ and $|N_H(y) \cap N_H(z)| \geq \gamma n^{k-1}$.
\end{itemize}
There there are at least $\gamma ' n^{2k-1}$ $(2k-1)$-sets $X \subseteq V$ such that both $H[X \cup \{x\}]$ and $H[X \cup \{y\}]$ contain perfect matchings.
\end{lemma}

\begin{lemma} \label{lem:main}
Let $k\in \mathbb N$ and  $0< \a \ll \gamma \ll \eps , 1/k$.
Then there exists an $n_0 \in \mathbb N$ such that the following holds. Let $H=(V, E)$ be a $k$-uniform hypergraph on $n \geq n_0$ vertices such that $\delta_1(H)\ge (1/2 - \a) \binom{n-1}{k-1}$.
Suppose that there exists $x_0, y_0\in V$ such that
\begin{itemize}
\item[(i)] $|N_H (x_0) \cap N_H (y_0) | < \gamma n^{k-1}$;
\item[(ii)] at most $\gamma n$ vertices $z \in V$ satisfy $|N_H(z) \cap N_H(x_0)| \geq \gamma n^{k-1}$ and $|N_H(z) \cap N_H(y_0)| \geq \gamma n^{k-1}$.
\end{itemize}
Then $H$ is $\eps$-close to $\mathcal B_{n,k}$ or $\overline{\mathcal B}_{n,k}$.
\end{lemma}

We postpone the proof of Lemmas~\ref{absorbing} and \ref{lem:main} and prove Theorem~\ref{thm:abs} first.

\smallskip

{\noindent \bf Proof of Theorem~\ref{thm:abs}.} Given $\eps >0$ and $k \geq 2$, choose constants $\alpha , \gamma ', \gamma$ so that
$0<\alpha \ll \gamma ' \ll \gamma \ll \eps, 1/k$. Set $\xi := (\gamma '/2)^k / \sqrt{128k(k-1)^2}$. Let $n$ be sufficiently large and $H$ be a $k$-uniform hypergraph as in the statement of the theorem.

By Lemmas~\ref{absorbing} and~\ref{lem:main},  $H$ is $\eps$-close to $\mathcal B_{n,k}$ or $\overline{\mathcal B}_{n,k}$ or for every $x, y\in V(H)$ there are at least $\gamma ' n^{2k-1}$ $(2k-1)$-sets $X \subseteq V(H)$ such that both $H[X \cup \{x\}]$ and $H[X \cup \{y\}]$ contain perfect matchings. In the former case we are done. In the latter case, Lemma~\ref{lo} implies that $H$ contains a matching $M$ so that
\begin{itemize}
\item $|M|\leq (\gamma '/2)^k n/(8k^2(k-1)) \leq\xi n/k$;
\item $M$ absorbs any $W \subseteq V(H) \setminus M$ such that $|W| \in k \mathbb N$ and  $|W|\leq  (\gamma '/2)^{2k} n/(128k(k-1)^2) =\xi ^2 n $,
\end{itemize}
as required.
\endproof

\noindent
{\bf Proof of Lemma~\ref{absorbing}.}
Note that (i) and (ii) imply that $\delta _1 (H) \geq \gamma n^{k-1}$ and
so $e(H) \geq \gamma n^k/k$.
Consider any $x,y \in V$. First assume that (i) holds. Fix any $X' \subseteq V$ where $|X'|=k-1$ and $X' \cup \{x\}, X' \cup \{y\} \in E$. By (i) there are at least $\gamma n^{k-1}$ choices for $X'$.
Next choose some $X'' \subseteq V\setminus (X' \cup \{x,y\})$ such that $|X''|=k$ and $X''\in E$. There are at least $\gamma n^k/k -(k+1) \binom{n}{k-1} \geq \gamma n^k /(2k)$ choices for $X''$.
Set $X:=X' \cup X''$. Note that both $H[X \cup \{x\}]$ and $H[X \cup \{y\}]$ contain perfect matchings. Further, since there are at least
$\gamma n^{k-1}$ choices for $X'$, at least $ \gamma n^k /(2k)$ choices for $X''$ and each $(2k-1)$-set may be counted at most $\binom{2k-1}{k-1}$  times, there are at least
\[
\gamma n^{k-1} \times \frac{\gamma n^k }{2k} \times \frac{1}{ \binom{2k-1}{k-1}} > \gamma ' n^{2k-1}
\]
choices for $X$ (as $\gamma ' \ll \gamma \ll 1/k$), as desired.

Now suppose that (ii) holds. Fix any $z \in V$ such that $|N_H(x) \cap N_H(z)| \geq \gamma n^{k-1}$ and $|N_H(y) \cap N_H(z)| \geq \gamma n^{k-1}$. There are at least $\gamma n $ choices for $z$. Next fix some $X' \in N_H(x) \cap N_H(z)$ that is disjoint from $y$. There are at least
 $\gamma n^{k-1}-\binom{n}{k-2} \geq \gamma n^{k-1}/2$ choices for $X'$. Finally, fix some $X'' \in N_H(y) \cap N_H(z)$ so that $X''$ is disjoint from $X' \cup \{x\}$.
 There are at least $\gamma n^{k-1} -k \binom{n}{k-2} \geq \gamma n^{k-1} /2$ choices for $X''$. Set $X:=X' \cup X''\cup \{z\}$.
 So $|X|=2k-1$ and both $H[X \cup \{x\}]$ and $H[X \cup \{y\}]$ contain perfect matchings.
 Further, there are at least
 $$\gamma n \times \frac{\gamma n^{k-1}}{2} \times \frac{\gamma n^{k-1}}{2} \times \frac{1}{(2k-1)\binom{2k-2}{k-1}} > \gamma ' n^{2k-1}$$ choices for $X$ (as $\gamma ' \ll \gamma \ll 1/k$), as desired.
\endproof

The rest of this section is devoted to the proof of Lemma~\ref{lem:main}. We draw on ideas used in the proof of Lemma 5.4 in~\cite{zhao2}. We need two results from \cite{zhao2}.  The first one implies that if any two vertices in a hypergraph have roughly the same neighborhood, then the hypergraph is near complete or empty.

\begin{lemma} \cite[Lemma 2.2]{zhao2}
\label{lem:KK}
Given any $k \in \mathbb N$ and $\rho >0$ there exists an $n_0 \in \mathbb N$ such that the following holds.
Let $F=(V, E)$ be a $k$-uniform hypergraph on $n \geq n_0$ vertices with edge density $|E|/\binom{n}{k}\in [\rho, 1 - \rho]$.
Then there exist two vertices $v, v'\in V$ such that $| N_F(v)\triangle N_F(v') | \ge \rho (1- \rho) n^{k-1}/(k+1)!$.
\end{lemma}

\begin{prop}\cite[Proposition 2.3]{zhao2}
\label{evensum}
For $r \in \mathbb N$, $0\le c\le 1$ and $n\to \infty$,
\begin{align*}
\sum_{0\le i\le r, \, i \, \rm{even}} \binom{cn}{r- i} \binom{ (1-c)n }{i} & = \frac{n^r}{2r!} ( 1+ (2c-1)^r ) - O(n^{r-1}), \\
\sum_{0\le i\le r, \, i \,\rm{odd}} \binom{cn}{r- i} \binom{ (1-c)n }{i} & = \frac{n^r}{2r!} ( 1- (2c-1)^r ) - O(n^{r-1}).
\end{align*}
\end{prop}

\smallskip

\noindent
{\bf Proof of  Lemma~\ref{lem:main}.}
Define
\[
X := \{ v\in V: |N_H(y_0) \cap N_H(v)| < \gamma n^{k-1} \} \quad \text{and} \quad
Y := \{v\in V: |N_H(x_0) \cap N_H(v)| < \gamma n^{k-1} \}.
\]
Then by Lemma~\ref{lem:main} (i), $x_0\in X$ and $y_0\in Y$. Let $V_0:= V\setminus (X\cup Y)$. We have $|V_0| \le \gamma n$ by Lemma~\ref{lem:main} (ii).
Roughly speaking, our goal is to show that $|X| \approx |Y| \approx n/2$ and $H \approx \dB(X, Y)$ or $H \approx \doB(X, Y)$.

We first provide several properties of $X$ and $Y$, for example, $X\cap Y = \emptyset$, and
$N_H(v)\approx N_H(v')$ whenever $v, v'\in X$ or $v, v'\in Y$.
\begin{claim}\label{clm:vv'} The following conditions hold.
\begin{enumerate}[{\rm (i)}]
\item For all $v\in X\cup Y$, we have $d_H(v) \le (1/2 + \a) \binom{n-1}{k-1} + \g n^{k-1}$.
\item $X\cap Y = \emptyset$.
\item For any two vertices $x, x' \in X$, we have $|N_H(x) \triangle N_H(x')| < 5\g n^{k-1}$. The same holds for all $y, y' \in Y$.
\item For any $x\in X$ and $y\in Y$, we have $|N_H(x) \cap N_H(y)| \le 4\g n^{k-1}$ and $|\overline{N}_H(x) \cap \overline{N}_H(y)| \le 4\g n^{k-1}$, where $\overline{N}_H(x) := \binom{V\setminus \{x\} }{ k-1}\setminus N_H(x)$ consists of \emph{non-neighbors} of $x$.
\end{enumerate}
\end{claim}

\begin{proof}
To see (i), suppose $x\in X$ and $y\in Y$. Since $d_H(x_0), d_H(y_0)\ge (1/2 - \a) \binom{n-1}{k-1}$, by the definition of $X$ and $Y$,
\begin{align} \label{eq:Nx0}
|N_H(y_0)\setminus N_H(x)|, |N_H(x_0)\setminus N_H(y)| \ge \left(\frac12 - \a \right) \binom{n-1}{k-1} - \g n^{k-1}.
\end{align}
Consequently, we have $d_H(x), d_H(y) \le (1/2 + \a) \binom{n-1}{k-1} + \g n^{k-1}$.

\smallskip
To see  (ii), suppose that there exists $v\in X\cap Y$. Then by \eqref{eq:Nx0},
\begin{align*}
 |( N_H( x_0)\cup N_H(y_0) ) \setminus N_H(v) | &\ge | N_H(x_0)\setminus N_H(v) | + | N_H(y_0) \setminus N_H(v)| - | N_H(x_0) \cap N_H(y_0) |\\
  &\ge 2 \left(\frac{1}2 - \a \right) \binom{n-1}{k-1} - 3\gamma n^{k-1} \ge \binom{n-1}{k-1} - 4\gamma n^{k-1},
\end{align*}
which implies that $|N_H(v)|\le 4\gamma n^{k-1}$, contradicting the minimum degree condition of $H$.

\smallskip
To see  (iii), consider $x\in X$. By the definition of $X$ and the minimum degree condition of $H$, we have
$| N_H(x) \cup N_H(y_0)| \ge 2 \left(\frac{1}2 - \a \right) \binom{n-1}{k-1} - \gamma n^{k-1}$.
Let $x'\in X\setminus \{x\}$. Then
\begin{align}
| N_H(x') \setminus N_H(x) | & = | (N_H(x') \setminus N_H(x) )\setminus N_H(y_0) ) | + | (N_H(x') \setminus N_H(x)) \cap N_H(y_0)| \nonumber\\
&\le \left | \binom{V\setminus \{x'\} }{k-1}\setminus ( N_H(x) \cup N_H(y_0) ) \right | + | N_H(x')\cap N_H(y_0) | \nonumber\\
&\le \binom{n-1}{k-1} - 2 \left(\frac{1}2 - \a \right) \binom{n-1}{k-1} + \gamma n^{k-1} + \gamma n^{k-1} \nonumber\\
&=  2\a \binom{n-1}{k-1} + 2\gamma n^{k-1}.  \label{eq:difN}
\end{align}
The same bound holds for $|N_H(x) \setminus N_H(x') |$. Hence
\begin{align*}
|N_H(x) \triangle N_H(x')| &= | N_H(x') \setminus N_H(x) | +  |N_H(x) \setminus N_H(x') | \le 2\left( 2\a \binom{n-1}{k-1} + 2\gamma n^{k-1} \right) < 5\g n^{k-1}.
\end{align*}
Analogously we can derive that $|N_H(y) \triangle N_H(y')| < 5\g n^{k-1}$ for any $y, y' \in Y$.

\smallskip
To see  (iv), consider $x\in X$ and $y\in Y$. By \eqref{eq:difN}, we have $|N_H(x)\setminus N_H(x_0)| \le  2\a \binom{n-1}{k-1} + 2\gamma n^{k-1}$. By the definition of $Y$, we have $|N_H(y) \cap N_H(x_0)| < \gamma n^{k-1}$. Thus
\begin{align}\label{eq:comN}
| N_H(x) \cap N_H(y) | \le | N_H(y) \cap N_H(x_0) | + | N_H(x) \setminus N_H(x_0) | \le  2\a \binom{n-1}{k-1} + 3\gamma n^{k-1}  \le 4\g n^{k-1},
\end{align}
which proves the first assertion of  (iv). By the minimum degree condition and \eqref{eq:comN}, we have
\[
|N_H(x) \cup N_H(y)| \ge 2 \left(\frac{1}2 - \a \right) \binom{n-1}{k-1} - 3\g n^{k-1} - 2\a \binom{n-1}{k-1}.
\]
It follows that
\begin{align*}
|\overline{N}_H(x) \cap \overline{N}_H(y)| &\le \binom{|V\setminus \{x, y\}|}{k-1} - | N_H(x) \cup N_H(y) |\\
	&\le \binom{n-2}{k-1} - (1 - 2\a)\binom{n-1}{k-1} + 3\g n^{k-1} +  2\a \binom{n-1}{k-1} \le 4\g n^{k-1},
\end{align*}
which proves the second assertion of  (iv).
\end{proof}

Since $|V_0|\leq \gamma n$ and $X \cap Y =\emptyset$ we have $|X|\geq (1-\gamma)n/2$ or $|Y|\geq (1-\gamma )n/2$. Without loss of generality we may assume that
$|X|\geq (1-\gamma)n/2\ge n/3.$
Let $0< \g_0< 1/2$ such that $\frac{\g_0(1-\g_0)}{ (k+1)! } = 5\g \cdot 3^{k-1}$. We apply Lemma~\ref{lem:KK} to $F= H[X]$ with $\rho = \g_0$.  Since
\[
|N_F(v) \triangle N_F(v')|\leq |N_H(v) \triangle N_H(v')| < 5\g n^{k-1} \le \frac{\g_0(1-\g_0)}{ (k+1)! }|X|^{k-1}
\]
for any $v, v'\in X$ (Claim~\ref{clm:vv'} (iii)), there are two possible cases:
\begin{description}
  \item[Case 1] $e(H[X])  \le \g_0\binom{|X|}{k} $,
  \item[Case 2] $e(H[X]) \ge (1- \g_0)\binom{|X|}{k}$.
\end{description}
In the rest of the proof we assume that one of the two cases holds. Once we have obtained more information we will prove that $|X|$ and $|Y|$ are close to $n/2$. At present we require the following weaker lower bounds on $|X|$ and $|Y|$.

\begin{claim}\label{clm:XY}
$|X|,|Y| \ge (1 - (\frac12 + 2\g_0)^{\frac{1}{k-1}} - \g) n$.
\end{claim}

\begin{proof} The bound on $|X|$ follows since $|X|\geq (1-\gamma)n/2$.
Since $|X|+ |Y| + |V_0| = n$ and $|V_0| \le \g n$, to prove the bound on $|Y|$, it suffices to show that $|X|\leq (\frac12 + 2\g_0)^{\frac{1}{k-1}} n$. In Case 1, there exists a vertex $x\in X$ such that $d_{H[X]}(x) \le \g_0 \binom{|X| -1}{k-1}$ and consequently, $|\overline{N}_{H[X]}(x)| \ge (1- \g_0) \binom{|X| -1}{k-1}$. Together with the minimum degree condition, this gives
\[
\left (\frac12 - \a \right ) \binom{n-1}{k-1} \le d_{H}(x) \le \binom{n-1}{k-1} - (1- \g_0) \binom{|X| -1}{k-1}.
\]
In Case 2, there exists a vertex $x\in X$ such that $d_{H[X]}(x) \ge (1- \g_0) \binom{|X| -1}{k-1}$. By Claim~\ref{clm:vv'} (i),
\[
(1- \g_0) \binom{|X| -1}{k-1} \le d_{H[X]}(x) \le \left(\frac12 + \a \right) \binom{n-1}{k-1} + \g n^{k-1}.
\]
In either case we have $(1- \g_0) \binom{|X| -1}{k-1} \le (\frac12 + \a) \binom{n-1}{k-1} + \g n^{k-1}$, which implies that $\binom{|X| -1}{k-1} \le (1+ 2\g_0)((1/2 + \a) \binom{n-1}{k-1} + \g n^{k-1})$. Letting $|X| = c n$, it follows that
\[
c^{k-1} \binom{n-1}{k-1} - O(n^{k-2}) \le \left(\frac12 + \g_0 + 2\a \right) \binom{n-1}{k-1}+ 2\g n^{k-1}
\]
Since $\g n^{k-1} \le \frac{\g_0}{5} \binom{n-1}{k-1}$, we conclude that $c\le (\frac12 + 2\g_0)^{\frac{1}{k-1}}$.
\end{proof}

Given two disjoint subsets $A, B\subset V$ and two integers $i, j\ge 0$, we call an $(i+j)$-set $S\subseteq V$ an $A^i B^j$-set if $|S\cap A|=i$ and $|S\cap B|=j$, and let $A^i B^j$ denote the family of all $A^i B^j$-sets.
Let $c_0 := 1 - (\frac12 + 2\g _0)^{\frac{1}{k-1}} - \g $, and $\g_i := \g_{i-1} + 5\g k!/ c_0^{k-1}$ for $i=1, \dots, k$.

\begin{claim} \label{clm:dXY}
\begin{enumerate}[{\rm (i)}]
\item In Case 1, for $1\le i\le k$, for any $y\in Y$, at least $(1- \g_i) \binom{|X|}{k-i} \binom{|Y|}{i-1}$ $X^{k-i} Y^{i-1}$-sets are neighbors (respectively, non-neighbors) of $y$ if $i$ is odd (respectively, even).
\item In Case 2, for $1\le i\le k$, for any $y\in Y$, at least $(1- \g_i) \binom{|X|}{k-i} \binom{|Y|}{i-1}$ $X^{k-i} Y^{i-1}$-sets are neighbors (respectively, non-neighbors) of $y$ if $i$ is even (respectively, odd).
\end{enumerate}
\end{claim}

\begin{proof} We prove both cases by induction on $i$. In Case 1 there exists a vertex $x_1\in X$ such that $d_{H[X]}(x_1) \le \g_0 \binom{|X| -1}{k-1}$ and consequently $|\overline{N}_H(x_1)\cap \binom{X}{k-1}| \ge (1 - \g_0) \binom{|X| -1}{k-1}$.
Fix a vertex $y\in Y$. By Claim~\ref{clm:vv'} (iv), $|\overline{N}_H(x_1) \cap \overline{N}_H(y)| \le 4\g n^{k-1}$. Thus at least $(1 - \g_0) \binom{|X| -1}{k-1} - 4\g n^{k-1}$ $X^{k-1}$-sets are neighbors of $y$. By Claim~\ref{clm:XY}, $|X|, |Y|\ge c_0 n$. Then for any $0\le i\le k-1$,
\begin{equation}\label{eq:XYi}
    \binom{|X|}{k-i-1} \binom{|Y|}{i} \ge \frac{(c_0 n)^{k-1}}{(k-1)!} - O(n^{k-2}) \ge \frac{c_0^{k-1} n^{k-1}}{k!}.
\end{equation}
Together with the definition of $\g_1$, we conclude that at least
\[
 (1 - \g_0) \binom{|X| -1}{k-1} - 4\g n^{k-1} \ge (1 - \g_0) \binom{|X|}{k-1} - O(n^{k-2}) - \frac{4\g k!}{c_0^{k-1}} \binom{|X|}{k-1} \ge (1 - \g_1) \binom{|X|}{k-1}
\]
$X^{k-1}$-sets are neighbors of $y$. This confirms (i) for $i=1$. In Case 2,  by averaging, there exists a vertex $x_1\in X$ such that $d_{H[X]}(x_1) \ge (1- \g_0) \binom{|X| -1}{k-1}$.  Fix a vertex $y\in Y$. By Claim~\ref{clm:vv'} (iv), $| N_H(x_1) \cap N_H(y)| \le 4\g n^{k-1}$. Thus at least
$(1 - \g_0) \binom{|X| -1}{k-1} - 4\g n^{k-1} \ge (1 - \g_1) \binom{|X|}{k-1}$
$X^{k-1}$-sets are non-neighbors of $y$. This confirms (ii) for $i=1$.

For the induction step, we first assume that for some $1\le i\le k$, every $y\in Y$ has at least $(1- \g_i) \binom{|X|}{k-i} \binom{|Y|}{i-1}$ $X^{k-i} Y^{i-1}$-sets in its neighborhood. Consequently at least $(1- \g_i) \binom{|X|}{k-i} \binom{|Y|}{i}$  $X^{k-i} Y^{i}$-sets are edges of $H$. By averaging, there exists $x_i\in X$ whose neighborhood contains at least $(1- \g_i) \binom{|X| - 1}{k- i-1} \binom{|Y|}{i}$ $X^{k-i-1} Y^{i}$-sets. Fix $y\in Y$. By Claim~\ref{clm:vv'} (iv),
\[
|(N_H(x_i)\cap X^{k-i-1} Y^i) \setminus N_H(y)|\ge (1 - \g_i) \binom{|X| - 1}{k- i- 1} \binom{|Y|}{i} - 4\g n^{k-1} .
\]
Since $|\overline{N}_H(y)\cap X^{k-i-1} Y^i| \ge |(N_H(x_i)\cap X^{k-i-1} Y^i) \setminus N_H(y)| - O(n^{k-2})$, we conclude that at least
\[
(1 - \g_i) \binom{|X| - 1}{k- i- 1} \binom{|Y|}{i} - 4\g n^{k-1} - O(n^{k-2})\ge (1 - \g_{i+1}) \binom{|X|}{k-i -1} \binom{|Y|}{i} 
\]
$X^{k-i-1} Y^{i}$-sets are non-neighbors of $y$, where we use \eqref{eq:XYi} and the definition of $\g_{i+1}$.
Analogously we can show that if for some $1\le i\le k$, every $y\in Y$ has at least $(1- \g_i) \binom{|X|}{k-i} \binom{|Y|}{i-1}$ $X^{k-i} Y^{i-1}$-sets as  non-neighbors, then at least $(1 - \g_{i+1}) \binom{|X|}{k-i -1} \binom{|Y|}{i}$ $X^{k-i-1} Y^{i}$-sets are neighbors of $y$.
This completes our induction proof.
\end{proof}

\begin{claim}\label{clm:n2}
$|X|, |Y| \ge (1- \eta) n/2$, where $\eta :=  (2\g_k)^{1/(k-1)} + \g$.
\end{claim}

\begin{proof}
Suppose that Claim~\ref{clm:dXY}  (i) holds (the proof when  Claim~\ref{clm:dXY}  (ii) holds is analogous).

Let $\tilde{n} := |X\cup Y|$. Note that $\tilde{n} \ge (1- \gamma)n$ because $|V_0| \leq \gamma n$.
Let $c:= |X| / \tilde{n}$. It suffices to show that $(1 - (2\g_k)^{1/(k-1)})/2 \le c\le (1 + (2\g_k)^{1/(k-1)})/2 $ because this implies that
\[
|X| = c\tilde{n}\ge \frac12 \big(1 - (2\g_k)^{1/(k-1)} \big) (1- \g)n > (1 - \eta) \frac{n}{2}
\]
and $|Y| = (1-c)\tilde{n} \ge \frac12 (1 - (2\g_k)^{1/(k-1)}) (1- \g)n > (1 - \eta) {n}/{2}$.

For any $y\in Y$, by Claim~\ref{clm:dXY} (i),
\[
d_H(y) \ge \sum_{1\le i\le k, \, i \,\rm{odd} } (1- \g_i) \binom{|X| }{k-i} \binom{|Y|}{i-1} \ge (1-\g_k) \sum_{0\le j\le k-1, \, j \,\rm{even} } \binom{|X|}{k-1-j} \binom{|Y|}{j} .
\]
Hence, by Proposition~\ref{evensum}, ${d}_H(y) \ge  (1- \g_k) \frac{\tilde{n}^{k-1}}{2(k-1)!} ( 1+ (2c-1)^{k-1} ) -O(n^{k-2})$. If $(2c-1)^{k-1} \ge 2\g_k$, then
\[
d_H(y) \ge (1- \g_k)\frac{(1- \g)^{k-1} n^{k-1}}{2(k-1)!} (1+ 2\g_k) - O(n^{k-2}) \ge \left (1 + \frac{\g_k}2\right ) \frac{n^{k-1}}{2(k-1)!},
\]
as $5\g(k-1)! < \g_k \ll 1$.
This contradicts  Claim~\ref{clm:vv'} (i).
Thus $(2c-1)^{k-1} < 2\g_k$. If $c\ge 1/2$, then $c< (1+ (2\g_k)^{1/(k-1)})/2$; if $c< 1/2$ and $k-1$ is even, then $(1-2c)^{k-1} = (2c- 1)^{k-1} < 2\g_k$ and thus $c> (1 - (2\g _k)^{1/{(k-1)}})/2$. In either case we are done. Otherwise assume that $c< 1/2$ and $k-1$ is odd. By Claim~\ref{clm:dXY} (i),
\[
\overline{d}_H(y) \ge \sum_{1\le i\le k, \, i \,\rm{even} } (1- \g_i) \binom{|X| }{k-i} \binom{|Y|}{i-1} \ge (1-\g_k) \sum_{0\le j\le k-1, \, j \,\rm{odd} } \binom{|X|}{k-1-j} \binom{|Y|}{j},
\]
where $\overline{d}_H(y):= |\overline{N} _H (y)|$.
By Proposition~\ref{evensum}, we have
$\overline{d}_H(y) \ge (1- \g_k) \frac{\tilde{n}^{k-1}}{2(k-1)!} ( 1- (2c-1)^{k-1} ) -O(n^{k-2})$.
If $1- (2c-1)^{k-1} = 1 + (1-2c)^{k-1} \ge 1 + 2\g_k$, then we obtain a contradiction as before  because $ \overline{d}_H(y) \le(\frac12 + \a)\binom{n-1}{k-1}$. Hence, $(1- 2c)^{k-1} < 2\g_k$ and consequently $c> (1 - (2\g_k)^{1/(k-1)})/2$, as desired.
\end{proof}

By Claim~\ref{clm:n2}, there exists a partition $ X', Y'$ of $V$ such that $|X'|= \lceil n/2 \rceil$, $|Y' | =  \lfloor n/2 \rfloor$ and $| X\cap X'|, |Y\cap Y'| \ge (1-\eta) n/2$. We claim that $H$ is $\eps$-close to $\dB(Y', X')$ in Case~1 and $\eps$-close to $\doB(Y', X')$ in Case~2.
Indeed, set $\B := \B_{\tilde{n}, k}(Y, X)$, where $\tilde{n} := |X\cup Y|$, and $H' := H[X\cup Y]$.
By definition, $E(\B)$ consists of all $X^{k-i} Y^{i}$-sets for all odd $0\le i\le k$. If Claim~\ref{clm:dXY} (i) holds, then $|E(\B)\cap E(H')| \ge (1- \g_k) |E(\B)| $ and $|E(\oB)\cap E(\overline{H'})| \ge (1- \g_k) |E(\oB)| $. Thus
\[
| E(\B) \triangle E(H') | =  | E(\B)\setminus E(H')| + | E(\oB)\setminus E(\overline{H'})| \le \g_k |E(\B)|  + \g_k |E(\oB)|\le \g_k \binom{n}{k}.
\]

Let $V' := (X\cap X')\cup (Y\cap Y')$. Then $|V'|\ge (1- \eta)n$ and
\[
\left| \big(E(\mathcal B_{n,k}(Y', X')) \triangle E(H) \big) \setminus \binom{V'}{k} \right| \le \eta n \binom{n-1}{k-1}.
\]
Since $\gamma \ll \eps ,1/k$, we have that $\g _k , \eta \ll \eps$.
Therefore,
\[
| E(\mathcal B_{n,k}(Y', X')) \triangle E(H)| \le  {\eta} n \binom{n-1}{k-1} + | E(\B) \triangle E(H') | \le {\eta} n \binom{n-1}{k-1} + \g_k \binom{n}{k}\le \eps \binom{n}{k}.
\]
which implies that $H$ is $\eps$-close to $\mathcal B_{n,k}(Y', X')$.
Analogously we can show that $H$ is $\eps$-close to $\overline{\mathcal B}_{n,k}(Y', X')$ in Case~2.
This completes the proof of Lemma~\ref{lem:main}.
\endproof

\section{An application of Lemma~\ref{absorbing}}
The following simple application of Lemma~\ref{absorbing} implies that the minimum $\ell$-degree condition that forces a perfect fractional matching also forces a perfect matching in a $k$-uniform hypergraph $H$, if we additionally assume that $H$ has a small number of vertices of large degree.

\begin{thm}\label{conc}
Given any $0<\eps \le \delta'$ and $k, \ell \in \mathbb N$ where $\ell<k$, there is an $n_0 \in \mathbb N$ such that the following holds. Let $H$ be a $k$-uniform hypergraph on $n\geq n_0$ vertices where $k$ divides $n$. Suppose that $\delta _{1} (H)\geq \delta' \binom{n-1}{k-1}$ and $\delta_{\ell}(H) \ge (c^*_{k, \ell} + \eps) \binom{n-\ell}{k-\ell}$. If there are at least $\eps n$ vertices $x \in V(H)$ so that  $d_H (x) \geq (1-\delta '+\eps)\binom{n-1}{k-1}$ then $H$ contains a perfect matching.
\end{thm}
\noindent
\textbf{Sketch proof.}
 It is easy to see that $H$ satisfies  Lemma~\ref{absorbing} (ii) (where we choose $0<\gamma \ll \eps$) and so by Lemma~\ref{lo}, $H$ contains a small absorbing matching $M$. Let $H':=H\setminus V(M)$. Then
$\delta _{\ell}(H') \geq (c^*_{k, \ell} + \eps/2) \binom{n-\ell}{k-\ell}$ and so by Lemma~\ref{lem:frac}, $H'$ contains a matching covering all but a very small set of vertices. After absorbing the uncovered vertices
 by $M$, we obtain a perfect matching in $H$.
\endproof

\section*{Acknowledgements}
This research was partially carried out whilst the  authors were visiting the Institute for Mathematics and its Applications at the University of Minnesota.
The authors would like to thank the institute for the nice working environment. The authors are also grateful
to  the referees for their  careful reviews.

{\footnotesize \obeylines \parindent=0pt

\begin{tabular}{lll}
Andrew Treglown                     &\ &  Yi Zhao \\
School of Mathematics						    &\ &  Department of Mathematics and Statistics \\
University of Birmingham   					&\ &  Georgia State University \\
Birmingham                          &\ &  Atlanta \\
B15 2TT															&\ &  Georgia 30303\\
UK																	&\ &  USA
\end{tabular}
}

{\footnotesize \parindent=0pt

\it{E-mail addresses}:
\tt{a.c.treglown@bham.ac.uk}, \tt{yzhao6@gsu.edu}}
\end{document}